\documentclass[11pt]{amsart}
\usepackage{amssymb,amsmath,amsthm}
\usepackage{amsfonts}
\usepackage{enumerate}
\usepackage{dsfont}
\usepackage{cite}
\usepackage[colorlinks, citecolor=red]{hyperref}
\usepackage{mathrsfs}
\usepackage{epsfig}
\usepackage{lscape}
\usepackage{subfigure}
\usepackage{epstopdf}
\usepackage{caption}
\usepackage{algorithm}
\usepackage{algpseudocode}
\usepackage{multirow}
\usepackage{geometry}
\usepackage{paralist}
\usepackage{enumerate}
\usepackage{url}
\usepackage{graphicx,cite}
\usepackage{longtable}
\usepackage{tikz}

\newtheorem{definition}{Definition}[section]
\newtheorem{corollary}[definition]{Corollary}

\newtheorem{theorem}[definition]{Theorem}
\newtheorem{lemma}[definition]{Lemma}

\newtheorem{remark}[definition]{Remark}

\newtheorem{example}[definition]{Example}

\date{}

\begin{document}
\baselineskip 18pt
\bibliographystyle{plain}

\title[A simple linear convergence analysis of  RRK]
{A simple linear convergence analysis of the randomized reshuffling Kaczmarz method}

\author{Deren Han}
\address{LMIB of the Ministry of Education, School of Mathematical Sciences, Beihang University, Beijing, 100191, China. }
\email{handr@buaa.edu.cn}

\author{Jiaxin Xie}
\address{LMIB of the Ministry of Education, School of Mathematical Sciences, Beihang University, Beijing, 100191, China. }
\email{xiejx@buaa.edu.cn}

\begin{abstract}
The random reshuffling Kaczmarz (RRK) method enjoys the simplicity and efficiency in solving linear systems as a Kaczmarz-type method, whereas it also inherits the practical improvements of the stochastic gradient descent (SGD) with random reshuffling (RR)  over original SGD. However, the current studies on RRK do not  characterize its convergence comprehensively. In this paper, we present a novel analysis of the RRK method and prove its linear convergence towards the unique least-norm solution of the linear system. Furthermore, the convergence upper bound is tight and does not depend on the dimension of the coefficient matrix.
\end{abstract}

\maketitle

\let\thefootnote\relax

\let\thefootnote\relax

\section{Introduction}

Solving systems of linear equations
\begin{equation}
	\label{main-prob}
	Ax = b, \qquad \text{where} \quad A \in \mathbb{R}^{m \times n}, \quad b \in \mathbb{R}^m,
\end{equation}
arises as a foundational problem in many fields of scientific computing and engineering, playing a critical role in optimal control \cite{Pat17}, signal processing \cite{Byr04}, machine learning \cite{Cha08}, and partial differential equations \cite{Ols14}.
Among the various methods for solving linear systems, the Kaczmarz method, which is also referred to as the algebraic reconstruction technique (ART), is renowned for its high efficiency and ease of implementation. 

The Kaczmarz method operates by selecting a row of the matrix in each iteration, and projecting the current solution estimate onto the hyperplane defined by that row, thereby refining the approximation iteratively. Empirical evidence in the literature suggests that selecting the rows of the matrix \(A\) in a random order, rather than a deterministic one, typically accelerates the convergence of the Kaczmarz method \cite{herman1993algebraic, natterer2001mathematics, feichtinger1992new}. Building on this idea,  Strohmer and Vershynin studied the randomized Kaczmarz (RK) method for consistent linear systems and proved its linear convergence in expectation in their seminal work \cite{Str09}. This breakthrough has inspired a large amount of research on the development of Kaczmarz-type methods, including accelerated RK methods \cite{liu2016accelerated, han2022pseudoinverse, loizou2020momentum, zeng2024adaptive}, randomized block Kaczmarz methods \cite{Nec19, needell2014paved, moorman2021randomized, Gow15, xie2024randomized}, randomized Douglas--Rachford methods \cite{han2024randomized}, greedy RK methods \cite{Bai18Gre, Gow19, su2023linear}, and randomized sparse Kaczmarz methods \cite{schopfer2019linear, chen2021regularized, zeng2023fast}, etc.

As a randomized approach, the RK method shares similar advantages with stochastic gradient descent (SGD) in addressing large-scale problems\cite{zeng2024adaptive,loizou2020momentum,garrigos2023handbook,gower2019sgd}. In fact, the RK method can be regarded as a variant of the stochastic gradient descent (SGD) method  \cite{robbins1951stochastic,needell2014stochasticMP,Str09} applied to the least-squares problem. See Section \ref{section-rrk} for more details. While a typical SGD iteration employs sampling without replacement to select a random gradient, a particularly effective variant uses sampling without replacement, also known as random reshuffling (RR)\cite{ahn2020sgd,mishchenko2020random,ying2018stochastic,nguyen2021unified}.
This sampling scheme introduces statistical dependence and eliminates the unbiased gradient estimation property inherent in SGD, which consequently complicates its theoretical analysis. Despite these challenges, RR has been empirically demonstrated to outperform original SGD in numerous practical applications \cite{sun2020optimization,gurbuzbalaban2021random,bottou2009curiously,recht2013parallel}, which is partly due to the simplicity and efficiency of implementing the random reshuffling sampling scheme, and the fact that RR utilizes all samples within each epoch. 

Applying the RR scheme to least squares problems results in the random reshuffling Kaczmarz (RRK) method. However, since the theoretical understanding of RR itself is mainly limited to in-expectation complexity bounds and almost sure asymptotic convergence results \cite{cha2023tighter,mishchenko2020random,haochen2019random,nguyen2021unified,rajput2020closing,safran2020good}, the existing convergence analysis for RRK either only focuses on the average case, 
or require additional assumption of a strongly convex objective function. See Section \ref{section-3.1} for more detailed discussions and insights into these results. 
Consequently, an interesting question arises: Is it possible to conduct a convergence analysis of the RRK method that does not rely on the current convergence framework of the RR method, but instead exploits the structure of the linear system itself? Furthermore, can this approach yield a superior convergence rate?

In this paper, we provide the first proof that the RRK method converges linearly to the unique least-norm solution, applicable to both full rank and rank-deficient coefficient matrices. Our convergence analysis treats the RRK method as a specific type of fixed-point iteration with dynamical iteration matrices, and we establish a uniform upper bound for the method by examining the properties of these matrices.
We further demonstrate that the convergence upper bound is tight, which means that there exists a linear system \(Ax = b\) for which the inequality for the upper bound holds with equality. 

\subsection{Notations}

For any matrix $A \in \mathbb{R}^{m \times n}$,  we use $a_{i,:},A^\top,A^\dagger,\|A\|_2$, $\mbox{Range}(A)$,  and $\text{Null}(A)$ to denote the $i$-th row, the transpose, the Moore-Penrose pseudoinverse, the spectral norm, the range space, and the null space of $A$, respectively. We use $\sigma_{\min}(A)$ to denote the smallest nonzero singular value of $A$. For any vector $b \in \mathbb{R}^m$, we use $b_i$ and $\|b\|_2$ to denote the $i$-th  entry and the Euclidean norm of $b$, respectively. The identity matrix is denoted by $I$. For any integer $m\geqslant 1$, we denote $[m]:=\{1,\ldots,m\}$. For any random variables $\xi_1$ and $\xi_2$, we use $\mathbb{E}[\xi_1]$ and $\mathbb{E}[\xi_1\lvert \xi_2]$ to denote the expectation of $\xi_1$ and the conditional expectation of $\xi_1$ given $\xi_2$. 


Throughout this paper, we use $x^*$ to denote an arbitrary solution of the linear system \eqref{main-prob}, and for any $x^0\in\mathbb{R}^n$, we set
$x^0_*:=A^\dagger b+(I-A^\dagger A)x^0
\
\text{and}
\ x_{LN}^*:=A^\dagger b.$
We mention that $x^0_*$
is the orthogonal projection of $x^0$ onto the set
$
\{x\in\mathbb{R}^n| A x= b\},
$
and $x_{LN}^*$ is the unique least-norm solution of the linear system.


\subsection{Organization}
The remainder of the paper is organized as follows.
In Section 2,  we briefly review the RR method and the RRK method. We analyze the RRK method and show its linear convergence rate in Section 3.
Finally, we conclude the paper in Section 4.

\section{Random reshuffling Kacmarz method }
\label{section-rrk}

First, we provide a brief introduction to the SGD method and the RR method. Consider the following unconstrained optimization problem where the objective function is the
sum of a large number of component functions
$$f(x)=\frac{1}{m}\sum_{i=1}^mf_i(x)$$
where $f_i:\mathbb{R}^n\to \mathbb{R}$. The SGD method is a popular approach for solving such large-scale problems. It employs the update rule 
$$x^{k+1}=x^k-\alpha_k\nabla f_{i_k}(x^k),$$
where $\alpha_k$ is the step-size and $i_k$ is selected randomly. This approach allows SGD to progress towards the minimum of the function using only a subset of the gradient information at each step, which is computationally advantageous, especially for large-scale problems. In the specific case where the objective function is
\begin{equation}\label{quar-f}
	f(x)=\frac{1}{2 m}\|A x-b\|_2^2=\frac{1}{m} \sum_{i=1}^m f_i(x),
\end{equation}
with $f_i(x)=\frac{1}{2}\left( \langle a_i,x \rangle -b_i\right)^2$,  the SGD method with a step-size $\alpha_k=1/\|a_{i_k}\|^2_2$ reduces to
\begin{equation}
	\label{RK-iteration}
	x^{k+1}=x^k-\frac{\langle a_{i_k},x^k\rangle-b_{i_k}}{\|a_{i_k}\|^2_2}a_{i_k},
\end{equation}
which is exactly the RK method \cite{Str09}.

In the context of large-scale classification problems, studies \cite{gurbuzbalaban2021random} have shown that utilizing a without-replacement sampling scheme in SGD can lead to faster convergence. This particular variant, known as Random Reshuffling (RR), is widely applicable in practice. In the \( k \)-th epoch of the RR method, indices \( \pi_{k,1}, \pi_{k,2}, \ldots, \pi_{k,m} \) are sampled without replacement from \([m]\), meaning \( \pi_k = (\pi_{k,1}, \pi_{k,2}, \ldots, \pi_{k,m}) \) is a random permutation of  \([m]\). Then an inner loop is conducted and the iterates are sequentially updated by 
\begin{equation}\label{SGD-RR}
	x^{k}_{i} = x^k_{i-1} - \lambda_{k,i} \nabla f_{\pi_{k,i}}(x^k_{i-1}), \quad i = 1, \cdots, m,
\end{equation}
where \( \lambda_{k,i} \) are appropriately chosen step-sizes. Next, set \( x^{k+1}=x^{k+1}_0 = x^k_m \) and proceed to the next epoch until the stopping criterion is met. We address that a new permutation (shuffle) is generated at the beginning of each epoch, thereby justifying the term  ``reshuffling''. 

When $f(x)$ is  of the least-squares type, as specified by \eqref{quar-f}, the RR method \eqref{SGD-RR} with the step-sizes $\lambda_{k,i}=1/\|a_{\pi_{k,i}}\|_2^2$ results in the RRK method. The detailed procedure for RRK is outlined in Algorithm \ref{rrk}. For simplicity and clarity, the algorithm is described in terms of  $a_{\pi_{k,1}},\ldots, a_{\pi_{k,m}}$ and $b_{\pi_{k}}=(b_{\pi_{k,1}},\ldots,b_{\pi_{k,m}})^\top$ instead of the gradient $\nabla f_{\pi_{k,i}}(x_i^k)$.

\begin{algorithm}[htpb]
	\caption{Random reshuffling Kacmarz method (RRK) \label{rrk}}
	\begin{algorithmic}
		\Require
		$A\in \mathbb{R}^{m\times n}$, $b\in \mathbb{R}^m$, $k=0$ and an initial $x^0\in \mathbb{R}^{n}$.
		\begin{enumerate}
			\item[1:] Set $x^{k}_0:=x^k$ and generate a random permutation $\pi_k=(\pi_{k,1},\pi_{k,2},\ldots,\pi_{k,m})$ of $[m]$.
			\item[2:] {\bf for $i=1,\ldots,m$ do}
			$$
			x_{i}^{k}:=x_{i-1}^k-\frac{\langle a_{\pi_{k,i}},x_{i-1}^k\rangle-b_{\pi_{k,i}}}{\|a_{\pi_{k,i}}\|^2_2}a_{\pi_{k,i}}.
			$$
			{\bf end for}
			\item[3:] Set
			$
			x^{k+1}:=x_{m}^k.
			$
			\item[4:] If the stopping rule is satisfied, stop and go to output. Otherwise, set $k=k+1$ and return to Step $1$.
		\end{enumerate}
		
		\Ensure
		The approximate solution.
	\end{algorithmic}
\end{algorithm}


We note that, as a byproduct of our analysis, our convergence results provide new insights for the shuffle-once and incremental variants of the Kaczmarz method (see Section \ref{sec:3.3}). 

\begin{itemize}
	\item {\bf Shuffle-once:} The shuffle-once algorithm \cite{mishchenko2020random,safran2020good} closely resembles the RR method, with the distinction that it shuffles the dataset only once at the start and uses this random permutation for all subsequent epochs. Formally, the indices \( \pi_{1}, \pi_{2}, \ldots, \pi_{m} \) are sampled without replacement from \([m]\) at the beginning, and for any \(k \geqslant 1\), we set \( \pi_k = (\pi_{1}, \pi_{2}, \ldots, \pi_{m}) \). 
	\item {\bf Incremental gradient:} The incremental gradient algorithm \cite{mishchenko2020random,safran2020good} is similar to shuffle-once, but the initial permutation is deterministic rather than random; that is, \( \pi_k = (1, 2, \ldots, m) \) for any \( k \geqslant 0 \).
\end{itemize}

When \(f(x)\) is of the least-squares type,  we refer to the shuffle-once algorithm and the incremental gradient algorithm with step sizes \(\lambda_{k,i}=1/\|a_{\pi_{k,i}}\|_2^2\) as the shuffle-once Kaczmarz (SOK) method and the incremental Kaczmarz (IK) method, respectively.

\section{Linear convergence of RRK}

In this section, we present our proof of the linear convergence of the RRK method. For convenience, we introduce some auxiliary variables.
Let $\pi_k=(\pi_{k,1},\pi_{k,2},\ldots,\pi_{k,m})$ be a permutation of $[m]$. Define 
\begin{equation}
	\label{proj_pi}
	T_{\pi_k}:=\left(I-\frac{a_{\pi_{k,m}}a^\top_{\pi_{k,m}}}{\|a_{\pi_{k,m}}\|^2_2}\right)\cdots	\left(I-\frac{a_{\pi_{k,1}}a^\top_{\pi_{k,1}}}{\|a_{\pi_{k,1}}\|^2_2}\right)
\end{equation}
and 
$$
g_{\pi_k}:=\sum_{i=1}^{m} \left(I-\frac{a_{\pi_{k,m}}a^\top_{\pi_{k,m}}}{\|a_{\pi_{k,m}}\|^2_2}\right)\cdots	\left(I-\frac{a_{\pi_{k,i+1}}a^\top_{\pi_{k,i+1}}}{\|a_{\pi_{k,i+1}}\|^2_2}\right)\frac{b_{\pi_{k,i}}}{\|a_{\pi_{k,i}}\|^2_2}a_{\pi_{k,i}}.
$$
Then the $k$-th epoch of the RRK method can be rewritten as
\begin{equation}
	\label{rrk_eq}
	x^{k+1}=T_{\pi_k}(x^k)+g_{\pi_k}.
\end{equation}
Since $T_{\pi_k}$ characterizes the transformation of the iterates throughout an entire epoch, we refer to it as the \textit{iteration matrix}. 

As $A^\dagger A$ is the orthogonal projector onto $\operatorname{Range}(A^\top)$, the following lemma illustrates that when the iteration matrix
$T_{\pi_k}$ is restricted to the range space of  $A^\top$, its spectral norm is less than $1$. In fact, this lemma can be derived from Theorem 3.7.4 in \cite{bauschke1997method}, which utilizes the concepts of regularity and strongly attracting mappings. For completeness, we here present a novel and straightforward proof.


\begin{lemma}
	\label{lemma-tpik}
	Assume that $T_{\pi_k}$ is defined as \eqref{proj_pi}. Then
	$$
	\|T_{\pi_k}A^\dagger A\|_2<1.
	$$
\end{lemma}
\begin{proof}
	The objective is to demonstrate that  for any $x\neq 0$, 
	$\|T_{\pi_k}A^\dagger Ax\|_2<\|x\|_2$. If $A^\dagger Ax=0$ the inequality is already satisfied. If $A^\dagger Ax\neq0$, then $A (A^\dagger Ax)\neq 0$, as $\text{Null}(A^\dagger)=\text{Null}(A^\top)$. Consequently, there exists a certain $i_0\in[m]$ such that $\langle a_{\pi_{k,i_0}},A^\dagger Ax \rangle\neq 0$, implying
	$$
	\left\|\left(I-\frac{a_{\pi_{k,i_0}}a^\top_{\pi_{k,i_0}}}{\|a_{\pi_{k,i_0}}\|^2_2}\right)A^\dagger Ax\right\|^2_2=\|A^\dagger Ax\|^2_2-\frac{\langle a_{\pi_{k,i_0}},A^\dagger Ax \rangle^2}{\|a_{\pi_{k,i_0}}\|^2_2}<\|A^\dagger Ax\|^2_2\leqslant \|x\|^2_2.
	$$
	Therefore, we obtain
	$$
	\begin{aligned}
		\|T_{\pi_k}A^\dagger Ax\|^2_2&=\left\|\left(I-\frac{a_{\pi_{k,m}}a^\top_{\pi_{k,m}}}{\|a_{\pi_{k,m}}\|^2_2}\right)\cdots	\left(I-\frac{a_{\pi_{k,i_0}}a^\top_{\pi_{k,i_0}}}{\|a_{\pi_{k,i_0}}\|^2_2}\right)A^\dagger Ax\right\|_2^2\\
		&\leqslant \left\|\left(I-\frac{a_{\pi_{k,m}}a^\top_{\pi_{k,m}}}{\|a_{\pi_{k,m}}\|^2_2}\right)\right\|^2_2\cdots	\left\|\left(I-\frac{a_{\pi_{k,i_0}}a^\top_{\pi_{k,i_0}}}{\|a_{\pi_{k,i_0}}\|^2_2}\right)A^\dagger Ax\right\|_2^2\\
		&\leqslant \left\|\left(I-\frac{a_{\pi_{k,i_0}}a^\top_{\pi_{k,i_0}}}{\|a_{\pi_{k,i_0}}\|^2_2}\right)A^\dagger Ax\right\|_2^2<\|x\|^2_2
	\end{aligned}
	$$
	as desired. This completes the proof of the lemma.
\end{proof}

\subsection{Convergence results for the RRK method}

We now present convergence results for Algorithm \ref{rrk}.
\begin{theorem}
	\label{main-THM}
	Suppose that  the linear system $Ax=b$ is consistent and $x^0\in\mathbb{R}^n$ is an arbitrary initial vector.
	Let $x^0_*=A^{\dagger}b+(I-A^\dagger A)x^0$. Then the iteration sequence $\{x^k\}_{k\geqslant 0}$ generated by Algorithm \ref{rrk} satisfies
	$$ \|x^{k+1}-x^{0}_*\|_2 
	\leqslant \|T_{\pi_k}A^\dagger A\|_2\cdot\|x^k-x^{0}_*\|_2,
	$$
	where $T_{\pi_k}$is defined as \eqref{proj_pi} and $\|T_{\pi_k}A^\dagger A\|_2<1$.
\end{theorem}

\begin{proof}
	According to Algorithm \ref{rrk}, one has $x^k\in x^0+\operatorname{Range}(A^\top)$. Besides, $ x^0_* =A^{\dagger}b+(I-A^\dagger A)x^0 = A^{\dagger}(b-Ax^0) + x^0 \in x^0 + \operatorname{Range}(A^\top)$.  Thus $ x^k - x_*^0 \in \operatorname{Range}(A^\top) $. Since  $A^\dagger A$ is the orthogonal projector onto $\operatorname{Range}(A^\top)$, one has 
	\begin{equation}
		\label{lemma-proof-0828-1}
		x^k-x^0_*=A^\dagger A(x^k-x^0_*).
	\end{equation}
	Therefore, it can be obtained from \eqref{rrk_eq} that
	\begin{equation}
		\label{for-remark}
		\begin{aligned}
			\|x^{k+1}-x^0_*\|_2&=\|T_{\pi_k}(x^k)+g_{\pi_k}-x^0_*\|_2\\
			&=\|T_{\pi_k}(x^k-x^0_*)\|_2
			\\
			&=\|T_{\pi_k}A^\dagger A(x^k-x^0_*)\|_2
			\\
			&\leqslant \|T_{\pi_k}A^\dagger A\|_2\cdot\|(x^k-x^0_*)\|_2,
		\end{aligned}
	\end{equation}
	where the second equality follows from  $x^0_*=T_{\pi_k}(x^0_*)+g_{\pi_k}$, and the third equality follows from \eqref{lemma-proof-0828-1}. It has been shown in Lemma \ref{lemma-tpik} that $\|T_{\pi_k}A^\dagger A\|_2<1$. This complete the proof of this theorem.
\end{proof}

Let $S_m$ denote the set of all permutations of $[m]$ and let
\begin{equation}
	\label{def-rho}
	\rho_{RRK}=\max_{\pi\in S_m} \|T_{\pi}A^\dagger A\|_2.
\end{equation}
Building on Theorem \ref{main-THM}, we derive the following corollary and demonstrate the linear convergence of Algorithm \ref{rrk}.

\begin{corollary}
	\label{main-rest}
	Under the same conditions of Theorem \ref{main-THM}, the iteration sequence 
	$\{x^k\}_{k\geqslant 0}$ generated by Algorithm \ref{rrk} satisfies
	$$ \|x^{k}-x^{0}_*\|_2 
	\leqslant \rho_{RRK}^k\|x^0-x^{0}_*\|_2,
	$$
	where $\rho_{RRK}$ is defined as \eqref{def-rho} and $\rho_{RRK}<1$.
\end{corollary}


Although our algorithm is randomized, it exhibits deterministic linear convergence, which may seem confusing. This contrasts with much of the literature on randomized iterative methods \cite{Str09,Gow15,han2024randomized,zeng2024adaptive}, where the focus is typically on the linear convergence of the expected error norm \(\mathbb{E}[\|x^k - x^*_0\|^2_2]\). The key reason is that our sampling space \(S_m\) is finite, allowing us to establish a uniform upper bound \(\rho_{RRK}\) in \eqref{def-rho}. In fact, deterministic linear convergence of \(\|x^k - x^*_0\|^2_2\) can result in lower iteration complexity compared to the linear convergence of \(\mathbb{E}[\|x^k - x^*_0\|^2_2]\). For further discussion, see \cite[Section 2.2]{su2023linear}.


\begin{remark}[Least-norm solution]
	If the initial vector $x^0\in\mbox{Range}(A^\top)$, then we have $x^0_*=A^{\dagger}b=x^*_{LN}$. This implies that the iteration sequence $ \{ x^k \}_{k\geqslant 0}$ generated by Algorithm \ref{rrk} now converges to the unique least-norm solution $x^*_{LN}$. 
\end{remark}

\begin{remark}[Tightness]
	Consider the matrix $A$ whose rows  satisfy the following conditions
	\[
	\langle a_i, a_j \rangle = \begin{cases}
		1 & \text{if } i = j, \\
		0 & \text{if } i \neq j.
	\end{cases}
	\]
	Then, for any permutation \(\pi_k\) of \([m]\), the matrix \(T_{\pi_{k}}\) in \eqref{proj_pi} simplifies to
	\[
	T_{\pi_{k}} = I - \sum_{i=1}^{m} a_i a_i^\top = I - A^\top A.
	\]
	Hence, we have
	\[
	T_{\pi_{k}} A^\dagger A = (I - A^\top A) A^\dagger A = A^\dagger A - A^\top A A^\dagger A = A^\dagger A - A^\top A = 0.
	\]
	This implies that the inequality in \eqref{for-remark} becomes an equality. Consequently, the upper bounds in Theorem \ref{main-THM} and Corollary \ref{main-rest} are also exact, indicating that these upper bounds are tight. In fact, for the linear system with this type of coefficient matrix, the RRK method can obtain the solution in a single step. 
\end{remark}

\subsection{Comparison to the existing convergence results for the RR method}
\label{section-3.1}

First, we restate some existing convergence results for the RR method in the context of least squares problems. We note that Theorems 2 and 3 in \cite{mishchenko2020random} were originally established for both strongly convex and convex problems. Here, we adapt these results to the least squares setting to enable a more direct comparison with our result.
\begin{theorem}[\cite{mishchenko2020random}, Theorem 2]
	\label{THM-Strongly}
	Suppose that the objective function \( f(x) \) is given by \eqref{quar-f} and the linear system \( Ax = b \) is consistent. If the coefficient matrix \( A \) is full column rank and the step-size \( \lambda_{k,i} = \gamma \) is a fixed constant satisfying \( \gamma \leqslant \frac{1}{\sqrt{2} m \|A\|^2_2} \), then the iteration sequence \( \{x^k\}_{k \geqslant 0} \) generated by the RR method \eqref{SGD-RR} satisfies 
	\[
	\mathbb{E}[\|x^k - A^\dagger b\|^2_2] \leqslant \left(1 - \frac{\gamma m \sigma_{\min}^2(A)}{2}\right)^k \|x^0 - A^\dagger b\|^2_2.
	\]
\end{theorem}

\begin{theorem}[\cite{mishchenko2020random}, Theorem 3]
	\label{THM-smooth}
	Suppose that the objective function \( f(x) \) is given by \eqref{quar-f} and the linear system \( Ax = b \) is consistent. Let \( \{x^k\}_{k \geqslant 0} \) be the sequence generated by the RR method \eqref{SGD-RR}. If the step-size \( \lambda_{k,i} = \gamma \) is a fixed constant satisfying \( \gamma \leqslant \frac{1}{\sqrt{2} m \|A\|^2_2} \), then the average iterate \( \hat{x}^k = \frac{1}{k} \sum_{i=1}^{k} x^i \) satisfies
	\[
	\mathbb{E}[f(\hat{x}^k)] \leqslant \frac{\|x^0 - x^*\|^2_2}{2 \gamma m k}.
	\]
\end{theorem}

Theorem \ref{THM-Strongly} shows that the RR method achieves linear convergence in expectation, and converges to the unique solution \(A^\dagger b\) of the linear system \(Ax=b\), when the coefficient matrix \(A\) is column full rank. Nevertheless, when the coefficient matrix \(A\) is not full rank, Theorem \ref{THM-smooth} only assures sub-linear convergence for the RR method, guaranteeing that the average iterate \(\hat{x}^k\) converges to an unknown solution of the linear system \(Ax=b\). However, Corollary \ref{main-rest} demonstrates that our linear convergence result is applicable to both full rank and rank-deficient coefficient matrices, with a tight convergence upper bound. Given an appropriate initial point, convergence to the unique least-norm solution can also be guaranteed.
In addition, the step-size for the RR method has to be constant, while the RRK method on the other hand, can adopt a dynamic step size   $\lambda_{k,i}=1/\|a_{\pi_{k,i}}\|^2_2$, which can be much larger than \( 1/\sqrt{2} m \|A\|^2_2 \). And a larger step-size usually implies higher computational efficiency. 

\subsection{Comparison of  RRK, SOK, IK, and RK}
\label{sec:3.3}

In this section, we compare the convergence upper bounds of  RK, RRK, SOK, and IK. In particular, we will present examples to illustrate their respective upper bounds.

We have previously established the convergence upper bound for RRK, denoted as $\rho_{RRK}$, in \eqref{def-rho}. Next, we briefly describe the convergence upper bounds for SOK, IK, and RK, respectively. 	 
Let the indices \( \pi_{1}, \pi_{2}, \ldots, \pi_{m} \) be sampled without replacement from \([m]\),  we set \( \pi_{SO} = (\pi_{1}, \pi_{2}, \ldots, \pi_{m}) \). It follows from Theorem  \ref{main-THM} that the SOK method with the random permutation \( \pi_{SO} \) exhibits the following convergence result
\[
\|x^{k}-x^*_0\|_2 \leqslant \rho_{SOK}^k  \|x^{0}-x^*_0\|_2,
\]
where $$\rho_{SOK}:=\|T_{\pi_{SO}}AA^\dagger\|_2.$$ Similarly, Theorem \ref{main-THM} shows that the IK method (\( \pi_{IK} = (1, 2, \ldots, m) \)) exhibits the following convergence result
\[
\|x^{k}-x^*_0\|_2 \leqslant \rho_{IK}^k  \|x^{0}-x^*_0\|_2,
\]
where $$\rho_{IK}:=\|T_{\pi_{IK}}AA^\dagger\|_2.$$
It has been proven \cite{Str09} that the RK method \eqref{RK-iteration} exhibits the following convergence result
$$
\mathop{\mathbb{E}} \left[\| x^{k}-x^*_0\|_2 \right] \leqslant \rho_{RK}^{k} \left\| x^0-x^*_0 \right\|_2,
$$
where $$\rho_{RK}:=\sqrt{1-\frac{\sigma_{\min}^{2} (A)}{\Vert A \Vert_F^2 }}.$$
Since the computational costs of the RRK method, the SOK method and the IK method at each epoch is about $m$-times as expensive as that of the RK method, we will account for this difference by considering $\rho^m_{RK}=\left(1-\frac{\sigma_{\min}^{2} (A)}{\Vert A \Vert_F^2 }\right)^{\frac{m}{2}}$ for the RK method.

By the definition of \(\rho_{RRK}\), we know that it represents the maximum value among all possible perturbations. Clearly, 
$$\rho_{RRK} \geqslant \rho_{IK} \ \ \text{and}\ \ \rho_{RRK} \geqslant \rho_{SO}.$$
However, if we consider only the convergence behavior within a single epoch, the RRK method may achieve a tighter convergence upper bound. The following example illustrates this point.

\begin{example}
	Consider the following coefficient matrix
	$$
	A=\begin{bmatrix}
		6 \ & \  4\\
		10  \ &  \ 4\\
		5 \  & \  8
	\end{bmatrix}.
	$$
	We have \(\|T_{(1,2,3)}\|_2 = \|T_{(3,2,1)}\|_2 \approx 0.7897\), \(\|T_{(3,1,2)}\|_2 = \|T_{(2,1,3)}\|_2 \approx 0.8918\), \(\|T_{(2,3,1)}\|_2 = \|T_{(1,3,2)}\|_2 \approx 0.7355\), and \(\rho_{RK}^3 \approx 0.8881\). It is evident that the convergence upper bounds of RRK, SOK, and IK are consistently better than that of RK.
	Furthermore, within a single epoch, the RRK method achieves the tightest convergence upper bound of $0.7355$ with  a probability of $1/3$.
\end{example}

The example above is artificially  constructed to illustrate the comparison of convergence upper bounds for RRK, SOK, IK, and RK. Below, we further compare these methods using real-world datasets.


\begin{example}
	The real-world data are obtained from the SuiteSparse Matrix Collection \cite{Kol19}.
	Each dataset includes  a matrix $A\in\mathbb{R}^{m\times n}$ and a vector $b\in\mathbb{R}^m$. In our experiments, we only use the matrices $A$ of the datasets and ignore the vector $b$. Specifically, we first generate the true solution $x^*={\tt randn(n,1)}$, and then compute $b=Ax^*$. All computations are initialized with $x^0=0$.  For each experiment, we run $20$ independent trials.

	Figure \ref{figureR1} illustrates the evolution of the relative solution error (RSE), defined as
	\[
	\text{RSE} = \frac{\|x^k - A^\dagger b\|_2^2}{\|x^0 - A^\dagger b\|_2^2},
	\]
	over the number of epochs for RRK, SOK, IK, and RK, and the worst-case convergence bounds derived from \(\rho_{IK}\) (Upper bound-IK) and \(\rho_{RK}\) (Upper bound-RK). Note that the worst-case convergence bounds derived from \(\rho_{RRK}\) and \(\rho_{SOK}\) are not plotted due to the computational impracticality of obtaining them. It can be seen that RRK and SOK are competitive compared to the other methods.
	Furthermore, the IK method performs the least effective, demonstrating the notable improvements in the Kaczmarz method brought by the randomization technique.
	
	\begin{figure}[tbhp]
		\centering
		\begin{tabular}{cc}
			\includegraphics[width=0.4\linewidth]{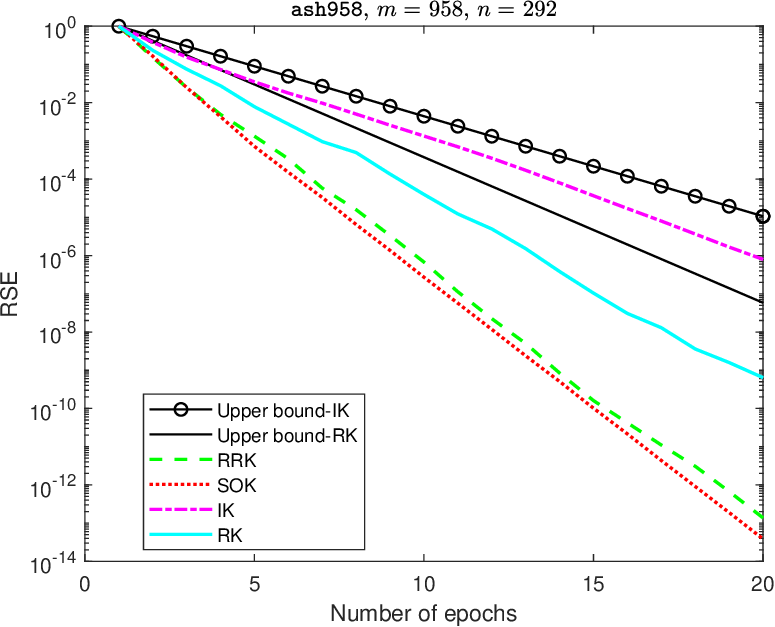}
			\includegraphics[width=0.4\linewidth]{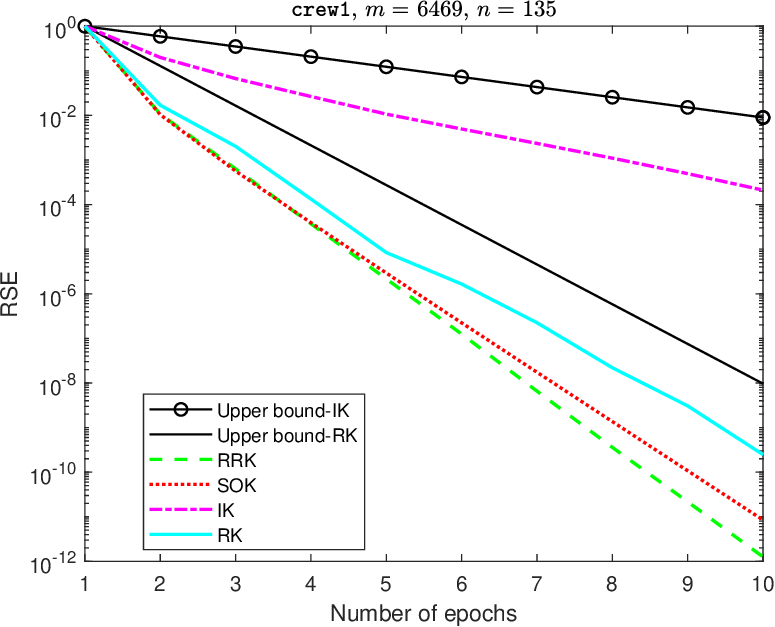}\\
			\includegraphics[width=0.4\linewidth]{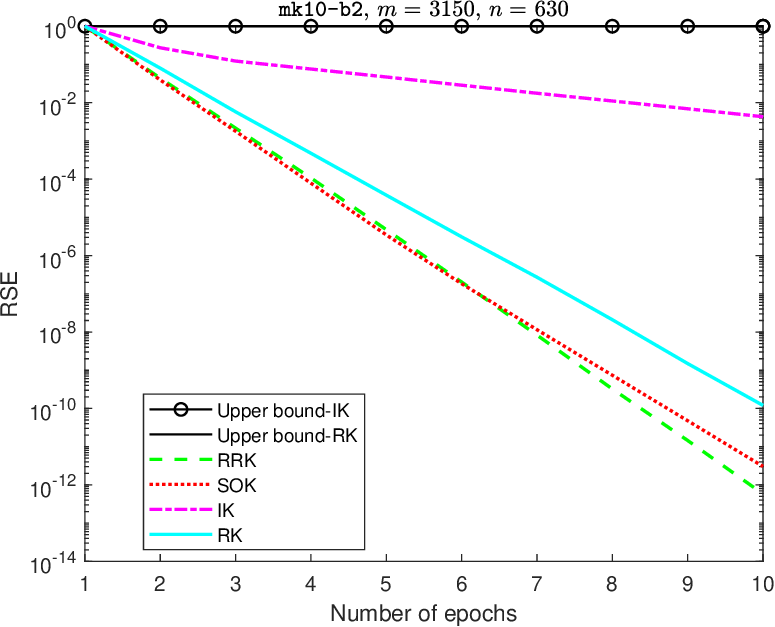}
			\includegraphics[width=0.4\linewidth]{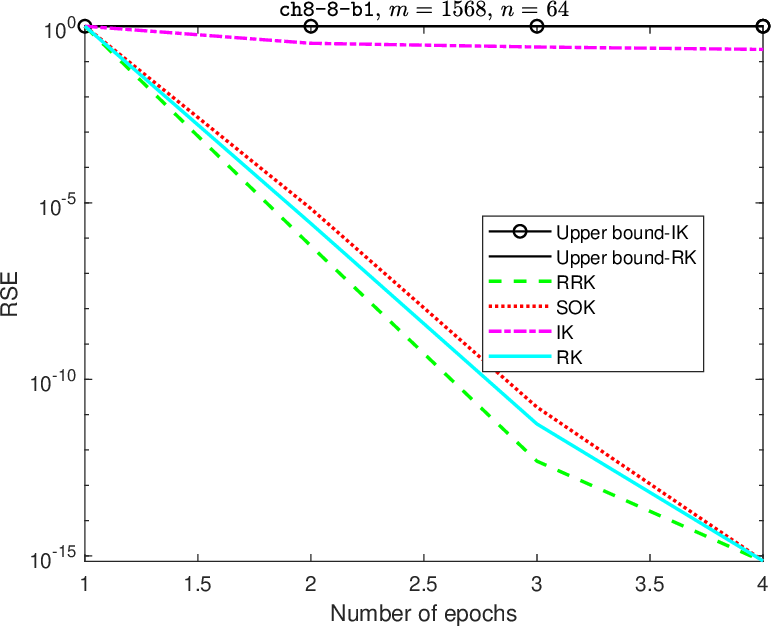}
		\end{tabular}
		\caption{The evolution of RSE with respect to the number of epochs. The title of each plot indicates the names and sizes of the data.}
		\label{figureR1}
	\end{figure}
	
\end{example}

\section{Concluding remarks}

We have established the linear convergence of the RRK method by analyzing the properties of the iteration matrices, and have shown that the convergence upper bound is tight. Moreover, our convergence analysis applies to both full rank and rank-deficient coefficient matrices.

Recent studies \cite{rieger2023generalized, hegland2023generalized} have shown that randomized Kaczmarz-type methods can be accelerated by the Gearhart-Koshy acceleration \cite{gearhart1989acceleration, tam2021gearhart}. They only proved that the resulting method converges to a certain solution of the linear system, without providing any convergence rate. The convergence analysis proposed in this paper could be beneficial for analyzing the Kaczmarz method with Gearhart-Koshy acceleration. Furthermore, the momentum acceleration technique is known for its effectiveness in improving optimization methods \cite{han2024randomized, loizou2020momentum, zeng2024adaptive}, it could be a valuable topic for exploring the momentum variant of the RRK method.

\bibliographystyle{abbrv}
\bibliography{main1230}

\begin{thebibliography}{10}

\bibitem{ahn2020sgd}
Kwangjun Ahn, Chulhee Yun, and Suvrit Sra.
\newblock {SGD} with shuffling: optimal rates without component convexity and
  large epoch requirements.
\newblock {\em Advances in Neural Information Processing Systems},
  33:17526--17535, 2020.

\bibitem{Bai18Gre}
Zhong-Zhi Bai and Wen-Ting Wu.
\newblock On greedy randomized {K}aczmarz method for solving large sparse
  linear systems.
\newblock {\em SIAM J. Sci. Comput.}, 40(1):A592--A606, 2018.

\bibitem{bauschke1997method}
Heinz~H Bauschke, Jonathan~M Borwein, and Adrian~S Lewis.
\newblock The method of cyclic projections for closed convex sets in hilbert
  space.
\newblock {\em Contemp. Math.}, 204:1--38, 1997.

\bibitem{bottou2009curiously}
L{\'e}on Bottou.
\newblock Curiously fast convergence of some stochastic gradient descent
  algorithms.
\newblock In {\em Proceedings of the symposium on learning and data science,
  Paris}, volume~8, pages 2624--2633. Citeseer, 2009.

\bibitem{Byr04}
Charles Byrne.
\newblock A unified treatment of some iterative algorithms in signal processing
  and image reconstruction.
\newblock {\em Inverse Problems}, 20(1):103--120, 2003.

\bibitem{cha2023tighter}
Jaeyoung Cha, Jaewook Lee, and Chulhee Yun.
\newblock Tighter lower bounds for shuffling {SGD}: {R}andom permutations and
  beyond.
\newblock In {\em International Conference on Machine Learning}, pages
  3855--3912. PMLR, 2023.

\bibitem{Cha08}
Kai-Wei Chang, Cho-Jui Hsieh, and Chih-Jen Lin.
\newblock Coordinate descent method for large-scale {L}2-loss linear support
  vector machines.
\newblock {\em J. Mach. Learn. Res.}, 9(7):1369–--1398, 2008.

\bibitem{chen2021regularized}
Xuemei Chen and Jing Qin.
\newblock Regularized {K}aczmarz algorithms for tensor recovery.
\newblock {\em SIAM J. Imaging Sci.}, 14(4):1439--1471, 2021.

\bibitem{feichtinger1992new}
Hans~Georg Feichtinger, C~Cenker, M~Mayer, H~Steier, and Thomas Strohmer.
\newblock New variants of the {POCS} method using affine subspaces of finite
  codimension with applications to irregular sampling.
\newblock In {\em Visual Communications and Image Processing'92}, volume 1818,
  pages 299--310. SPIE, 1992.

\bibitem{garrigos2023handbook}
Guillaume Garrigos and Robert~M Gower.
\newblock Handbook of convergence theorems for (stochastic) gradient methods.
\newblock {\em arXiv preprint arXiv:2301.11235}, 2023.

\bibitem{gearhart1989acceleration}
William~B Gearhart and Mathew Koshy.
\newblock Acceleration schemes for the method of alternating projections.
\newblock {\em J. Comput. Appl. Math.}, 26(3):235--249, 1989.

\bibitem{Gow19}
Robert~M Gower, Denali Molitor, Jacob Moorman, and Deanna Needell.
\newblock On adaptive sketch-and-project for solving linear systems.
\newblock {\em SIAM J. Matrix Anal. Appl.}, 42(2):954--989, 2021.

\bibitem{Gow15}
Robert~M. Gower and Peter Richtárik.
\newblock Randomized iterative methods for linear systems.
\newblock {\em SIAM J. Matrix Anal. Appl.}, 36(4):1660--1690, 2015.

\bibitem{gower2019sgd}
Robert~Mansel Gower, Nicolas Loizou, Xun Qian, Alibek Sailanbayev, Egor
  Shulgin, and Peter Richt{\'a}rik.
\newblock {SGD}: {G}eneral analysis and improved rates.
\newblock In {\em International conference on machine learning}, pages
  5200--5209. PMLR, 2019.

\bibitem{gurbuzbalaban2021random}
Mert G{\"u}rb{\"u}zbalaban, Asu Ozdaglar, and Pablo~A Parrilo.
\newblock Why random reshuffling beats stochastic gradient descent.
\newblock {\em Math. Program.}, 186:49--84, 2021.

\bibitem{han2024randomized}
Deren Han, Yansheng Su, and Jiaxin Xie.
\newblock Randomized {Douglas--Rachford} methods for linear systems: {I}mproved
  accuracy and efficiency.
\newblock {\em SIAM J. Optim.}, 34(1):1045--1070, 2024.

\bibitem{han2022pseudoinverse}
Deren Han and Jiaxin Xie.
\newblock On pseudoinverse-free randomized methods for linear systems:
  {U}nified framework and acceleration.
\newblock {\em arXiv preprint arXiv:2208.05437}, 2022.

\bibitem{haochen2019random}
Jeff Haochen and Suvrit Sra.
\newblock Random shuffling beats {SGD} after finite epochs.
\newblock In {\em International Conference on Machine Learning}, pages
  2624--2633. PMLR, 2019.

\bibitem{hegland2023generalized}
Markus Hegland and Janosch Rieger.
\newblock Generalized {Gearhart-Koshy} acceleration is a {Krylov} space method
  of a new type.
\newblock {\em arXiv preprint arXiv:2311.18305}, 2023.

\bibitem{herman1993algebraic}
Gabor~T Herman and Lorraine~B Meyer.
\newblock Algebraic reconstruction techniques can be made computationally
  efficient (positron emission tomography application).
\newblock {\em IEEE Trans. Medical Imaging}, 12(3):600--609, 1993.

\bibitem{Kol19}
Scott~P Kolodziej, Mohsen Aznaveh, Matthew Bullock, Jarrett David, Timothy~A
  Davis, Matthew Henderson, Yifan Hu, and Read Sandstrom.
\newblock The suitesparse matrix collection website interface.
\newblock {\em J. Open Source Softw.}, 4(35):1244, 2019.

\bibitem{liu2016accelerated}
Ji~Liu and Stephen Wright.
\newblock An accelerated randomized {K}aczmarz algorithm.
\newblock {\em Math. Comp.}, 85(297):153--178, 2016.

\bibitem{loizou2020momentum}
Nicolas Loizou and Peter Richt{\'a}rik.
\newblock Momentum and stochastic momentum for stochastic gradient, newton,
  proximal point and subspace descent methods.
\newblock {\em Comput. Optim. Appl.}, 77(3):653--710, 2020.

\bibitem{mishchenko2020random}
Konstantin Mishchenko, Ahmed Khaled, and Peter Richt{\'a}rik.
\newblock Random reshuffling: {S}imple analysis with vast improvements.
\newblock {\em Advances in Neural Information Processing Systems},
  33:17309--17320, 2020.

\bibitem{moorman2021randomized}
Jacob~D Moorman, Thomas~K Tu, Denali Molitor, and Deanna Needell.
\newblock Randomized {K}aczmarz with averaging.
\newblock {\em BIT.}, 61(1):337--359, 2021.

\bibitem{natterer2001mathematics}
Frank Natterer.
\newblock {\em The mathematics of computerized tomography}.
\newblock SIAM, 2001.

\bibitem{Nec19}
Ion Necoara.
\newblock Faster randomized block {K}aczmarz algorithms.
\newblock {\em SIAM J. Matrix Anal. Appl.}, 40(4):1425--1452, 2019.

\bibitem{needell2014stochasticMP}
Deanna Needell, Nathan Srebro, and Rachel Ward.
\newblock Stochastic gradient descent, weighted sampling, and the randomized
  {K}aczmarz algorithm.
\newblock {\em Math. Program.}, 155:549--573, 2016.

\bibitem{needell2014paved}
Deanna Needell and Joel~A Tropp.
\newblock Paved with good intentions: analysis of a randomized block kaczmarz
  method.
\newblock {\em Linear Algebra Appl.}, 441:199--221, 2014.

\bibitem{nguyen2021unified}
Lam~M Nguyen, Quoc Tran-Dinh, Dzung~T Phan, Phuong~Ha Nguyen, and Marten
  Van~Dijk.
\newblock A unified convergence analysis for shuffling-type gradient methods.
\newblock {\em J. Mach. Learn. Res.}, 22(207):1--44, 2021.

\bibitem{Ols14}
Maxim~A Olshanskii and Eugene~E Tyrtyshnikov.
\newblock {\em Iterative methods for linear systems: theory and applications}.
\newblock SIAM, Philadelphia, 2014.

\bibitem{Pat17}
Andrei Patrascu and Ion Necoara.
\newblock Nonasymptotic convergence of stochastic proximal point methods for
  constrained convex optimization.
\newblock {\em J. Mach. Learn. Res.}, 18(1):7204--7245, 2017.

\bibitem{rajput2020closing}
Shashank Rajput, Anant Gupta, and Dimitris Papailiopoulos.
\newblock Closing the convergence gap of {SGD} without replacement.
\newblock In {\em International Conference on Machine Learning}, pages
  7964--7973. PMLR, 2020.

\bibitem{recht2013parallel}
Benjamin Recht and Christopher R{\'e}.
\newblock Parallel stochastic gradient algorithms for large-scale matrix
  completion.
\newblock {\em Math. Program. Comput.}, 5(2):201--226, 2013.

\bibitem{rieger2023generalized}
Janosch Rieger.
\newblock Generalized {Gearhart-Koshy} acceleration for the kaczmarz method.
\newblock {\em Math. Comp.}, 92(341):1251--1272, 2023.

\bibitem{robbins1951stochastic}
Herbert Robbins and Sutton Monro.
\newblock A stochastic approximation method.
\newblock {\em Ann. Math. Statistics}, pages 400--407, 1951.

\bibitem{safran2020good}
Itay Safran and Ohad Shamir.
\newblock How good is {SGD} with random shuffling?
\newblock In {\em Conference on Learning Theory}, pages 3250--3284. PMLR, 2020.

\bibitem{schopfer2019linear}
Frank Sch{\"o}pfer and Dirk~A Lorenz.
\newblock Linear convergence of the randomized sparse {K}aczmarz method.
\newblock {\em Math. Program.}, 173(1):509--536, 2019.

\bibitem{Str09}
Thomas Strohmer and Roman Vershynin.
\newblock A randomized {K}aczmarz algorithm with exponential convergence.
\newblock {\em J. Fourier Anal. Appl.}, 15(2):262--278, 2009.

\bibitem{su2023linear}
Yansheng Su, Deren Han, Yun Zeng, and Jiaxin Xie.
\newblock On the convergence analysis of the greedy randomized {K}aczmarz
  method.
\newblock {\em arXiv preprint arXiv:2307.01988}, 2023.

\bibitem{sun2020optimization}
Ruo-Yu Sun.
\newblock Optimization for deep learning: An overview.
\newblock {\em J. Oper. Res. Soc. China}, 8(2):249--294, 2020.

\bibitem{tam2021gearhart}
Matthew~K Tam.
\newblock {G}earhart--{K}oshy acceleration for affine subspaces.
\newblock {\em Oper. Res. Lett.}, 49(2):157--163, 2021.

\bibitem{xie2024randomized}
Jiaxin Xie, Hou-Duo Qi, and Deren Han.
\newblock Randomized iterative methods for generalized absolute value
  equations: {S}olvability and error bounds.
\newblock {\em arXiv preprint arXiv:2405.04091}, 2024.

\bibitem{ying2018stochastic}
Bicheng Ying, Kun Yuan, Stefan Vlaski, and Ali~H Sayed.
\newblock Stochastic learning under random reshuffling with constant
  step-sizes.
\newblock {\em IEEE Trans. Signal Process.}, 67(2):474--489, 2018.

\bibitem{zeng2023fast}
Yun Zeng, Deren Han, Yansheng Su, and Jiaxin Xie.
\newblock Fast stochastic dual coordinate descent algorithms for linearly
  constrained convex optimization.
\newblock {\em arXiv preprint arXiv:2307.16702}, 2023.

\bibitem{zeng2024adaptive}
Yun Zeng, Deren Han, Yansheng Su, and Jiaxin Xie.
\newblock On adaptive stochastic heavy ball momentum for solving linear
  systems.
\newblock {\em SIAM J. Matrix Anal. Appl.}, 45(3):1259--1286, 2024.

\end{thebibliography}

\end{document}